\documentclass[11pt]{amsart}
\usepackage{mathtools}
\usepackage{amssymb}
\usepackage{cmap}           
\usepackage[margin=1in]{geometry}
\usepackage[utf8]{inputenc}
\usepackage{color}
\usepackage[pdfdisplaydoctitle=true,
            colorlinks=true,
            urlcolor=blue,
            citecolor=blue,
            linkcolor=blue,
            pdfstartview=FitH,
            pdfpagemode= UseNone,
            bookmarksnumbered=true]{hyperref}
\newtheorem{theorem}{Theorem}

\newtheorem{lemma}{Lemma}

\theoremstyle{definition}
\newtheorem{definition}{Definition}

\newtheorem{example}{Example}
\theoremstyle{remark}
\newtheorem{remark}{Remark}
\newcommand{\CC}{\mathbb{C}}

\newcommand{\RR}{\mathbb{R}}
\newcommand{\ZZ}{\mathbb{Z}}

\newcommand{\norm}[1]{\left\Vert#1\right\Vert}
\newcommand{\abs}[1]{\left\vert#1\right\vert}

\DeclareMathOperator{\sgn}{sign} %

\newcommand{\Diff}{\mathrm{Diff}}
\newcommand{\D}[1]{\mathcal{D}^{#1}(S^{1})}
\newcommand{\CS}{\mathrm{C}^{\infty}(S^{1})}

\newcommand{\HH}[1]{H^{#1}(S^{1})}

\mathtoolsset{showonlyrefs}

\hypersetup{
    pdftitle={Geodesic completeness of the H3/2 metric on Diff(S1)}, 
    pdfauthor={M. Bauer, B. Kolev and S.~C. Preston}, 
    pdfsubject={MSC 2010: 35Q35, 53D25}, 
    pdfkeywords={Euler--Arnold equation, Geodesic flows on the diffeomorphisms group, Sobolev metrics of fractional order, Global existence of solutions}, 
    pdflang=EN 
    }
\begin{document}

\title{Geodesic completeness of the $H^{3/2}$ metric on $\mathrm{Diff}(S^{1})$}%

\author[M. Bauer]{Martin Bauer}
\address[Martin Bauer]{Department of Mathematics, Florida State University, 32301 Tallahassee, USA}
\email{bauer@math.fsu.edu}

\author[B. Kolev]{Boris Kolev}
\address[Boris Kolev]{Université Paris-Saclay, ENS Paris-Saclay, CNRS, LMT - Laboratoire de Mécanique et Technologie, 94235, Cachan, France}
\email{boris.kolev@math.cnrs.fr}

\author[S.~C. Preston]{Stephen C. Preston}
\address[Stephen C. Preston]{Department of Mathematics, Brooklyn College and the Graduate Center, City University New
  York, 11210 New York, USA}
\email{Stephen.Preston@brooklyn.cuny.edu}


\date{\today}%
\begin{abstract}
  Of concern is the study of the long-time existence  of solutions to the Euler--Arnold equation of the right-invariant $H^{\frac{3}{2}}$-metric on the diffeomorphism group of the circle. In previous work by Escher and Kolev it has been shown that this equation admits long-time solutions if the order $s$ of the metric is greater than $\frac{3}{2}$, but the behaviour for the critical Sobolev index $s=\frac32$ has been left open. In this article we fill this gap by proving the analogous result also for the boundary case. We show that the behaviour is the same for all Sobolev metrics of order $\frac{3}{2}$ regardless of lower-order terms.
\end{abstract}

\maketitle

\section{Introduction}
\label{sec:intro}

In this article, we prove longtime existence for solutions of the geodesic initial value problem of the right invariant $H^{3/2}$-metric on the group of smooth diffeomorphisms on the circle. The interest in (fractional) order metrics on diffeomorphism groups is fuelled by their relations to various prominent PDEs of mathematical physics: In the seminal article~\cite{Arn1966}, Arnold showed in 1965 that Euler's equations for the motion of an incompressible, ideal fluid have a geometric interpretation as the geodesic equations on the group of volume preserving diffeomorphisms. Since then, an analogous result has been found for a whole variety of PDEs, including the inviscid Burgers equation, the Hunter--Saxton equation, the Camassa--Holm equation~\cite{CH1993,Kou1999}, and the modified Constantin--Lax--Majda (mCLM) equation~\cite{CLM1985,EKW2012}. Building on the pioneering work of Ebin and Marsden~\cite{EM1970}, these geometric interpretations have been used to obtain rigourous well-posedness and stability results for the corresponding PDEs~\cite{CK2003,Shk2000,Shk1998,MP2010,BHM2013,Gay2009a,MM2013}.

Motivated by the analysis on the mCLM equation, Escher and Kolev recently studied fractional order Sobolev metrics on the diffeomorphism group of the circle~\cite{EK2014,EK2014a}. In their investigations, they showed that the geodesic equation of the class of Sobolev metrics of order $s$ is locally well-posed if $s\geq\frac12$ and globally well-posed if $s>\frac32$. The question of global existence of solutions of the geodesic equation for the critical index $s=\frac32$ was left unanswered. Towards this direction, the third author and Washabaugh proved in~\cite{PW2016} that the Weil--Petersson metric on the universal Teichmüller space, which is of critical order $\frac32$, possesses smooth global solutions. In this article, we extend their analysis to obtain a global existence result for general Sobolev metrics of order $3/2$ on the diffeomorphism group of the circle and thus give a positive answer for the critical index. The analysis for the highest order terms follows as in~\cite{PW2016}, with the main difficulty of the present article being the definition of a good general class of operators that allow to also bound the energy of the lower order terms. We are able to achieve this result for general Fourier multipliers, that allow for a power series expansion towards negative infinity.

\subsection*{Metric completeness} In \cite{BEK2015}, it was shown that the Sobolev metric of order $s>\frac32$ extends smoothly to a \emph{strong} Riemannian metric on the group of Sobolev diffeomorphisms $\mathcal D^s(S^1)$. This allowed the authors to use results on strong right-invariant metrics to show that the metric is not only geodesically complete, but (as a metric on $\mathcal D^s(S^1)$) also metrically complete; see also~\cite{BV2017}. This result is not true anymore for the critical index $s=\frac32$, as $\mathcal D^s(S^1)$ is only a topological group for $s>\frac32$; the metric extends only to a smooth, \emph{weak} Riemannian metric
on the Sobolev completion $\mathcal D^q(S^1)$, for high enough $q>\frac32$. For the metric corresponding to the Euler-Weil-Petersson equation, which is of critical order $\frac32$, Gay-Balmaz and Ratiu \cite{GAYBALMAZ2015717}
found an interpretation as a strong Riemannian metric on a certain subgroup of all quasisymmetric homeomorphisms of the circle and used this to conclude that the velocity field remains in $H^{3/2}$. Although we believe that a similar approach would be feasible for general metrics of order $\frac32$ both the results and methods are entirely distinct from the approach in this paper. In particular their approach does not seem to yield geodesic completeness in the smooth category or on any group of diffeomorphisms $\mathcal D^q(S^1)$ for $q>\frac{3}{2}$.

\subsection*{Metrics of lower order} For certain examples of metrics of order $s<\frac32$, it has been shown that solutions of the Euler--Arnold equation can blowup in finite time. This includes the $L^{2}$-metric (Burgers equation), the $H^{1/2}$-metric (mCLM equation) \cite{PW2016,BKP2016} and the $H^1$-metric (Camassa--Holm equation)~\cite{CH1993,CE1998b}. We conjecture that blowup of solutions occurs for every metric of order $s<\frac32$. This result would provide a complete characterization for the solution behaviour of the geodesic equation of fractional order metrics on the group of diffeomorphisms. As we are not able to show this result at the present time, we leave this question open for future research.

\subsection*{Effect of lower order terms} The effect of lower-order terms can be significant: for example the Camassa-Holm and Hunter-Saxton equations have the same highest-order term in the metric, but all solutions of the latter end in finite time~\cite{Len2008} while some solutions of the former remain smooth for all time~\cite{CH1993}. On the other hand, we show that if all solutions remain smooth for one metric of order $\frac{3}{2}$, the same is true for all other metrics of the same leading order.

\section{Right invariant Sobolev metrics on $\Diff(S^1)$}
\label{sec:right-invariant-metrics}

Let $\Diff(S^1)$ denote the group of smooth and orientation preserving diffeomorphisms on the circle. The space $\Diff(S^1)$ is an open subset of the Fréchet manifold of all smooth functions $C^{\infty}(S^1,S^1)$ and thus, itself a Fréchet manifold. Furthermore, composition and inversion are smooth maps and $\Diff(S^1)$ is a Fréchet-Lie group, where the Lie algebra is the space of vector fields on the circle, equipped with the negative of the usual Lie-bracket on vector fields, \textit{i.e.}:
\begin{align*}
  [u,w] = u_x w - u w_x\;.
\end{align*}
See \cite{Ham1982} for more details on diffeomorphism groups as infinite dimensional Lie groups. Given an inner product on the space of vector fields we can extend this using right translations to obtain a right invariant metric on the diffeomorphism group.
Thus, to define a \emph{right-invariant} metric, it remains only to specify the inner product on the space of vector fields. The most natural choice is given by the standard $L^{2}$-inner product:
\begin{equation}
  \langle u,v \rangle_{L^{2}} = \int_{S^1} uv\; dx \;.
\end{equation}
This metric (which corresponds to the Burgers equation) has been studied in great detail and it has been shown, in particular, that:
\begin{enumerate}
  \item the induced geodesic distance of the metric is vanishing~\cite{MM2005};
  \item the exponential map is not a $C^1$ diffeomorphism~\cite{CK2002}.
\end{enumerate}

More generally, we consider an inner product on $C^{\infty}(S^1)$ which writes
\begin{equation*}
  \langle u,v \rangle_{A} = \int_{S^1} (Au)v\; dx \;.
\end{equation*}
where $A: C^{\infty}(S^1)\rightarrow C^{\infty}(S^1)$, the so-called \emph{inertia operator} is self-adjoint, with respect to the $L^{2}$-inner product \textit{i.e.}
\begin{equation*}
  \int_{S^1} (Au)v\; dx = \int_{S^1} u(Av)\; dx, \qquad \forall u,v\in C^{\infty}(S^1) \,,
\end{equation*}
and positive definite \textit{i.e.}
\begin{equation*}
  \int_{S^1} (Au)u\; dx>0, \qquad \forall u\in C^{\infty}(S^1) \,.
\end{equation*}
The corresponding right invariant metric on $\Diff(S^1)$ reads as
\begin{align*}
  G^A_{\varphi}(h,k) = \langle h\circ{\varphi^{-1}},k\circ{\varphi^{-1}} \rangle_{A} = \int_{S^1} (A(h\circ{\varphi^{-1}}))k\circ{\varphi^{-1}}\; dx,
\end{align*}
where $h,k\in T_{\varphi}\Diff(S^1)$. If we assume, furthermore, that $A$ is invertible and commutes with differentiation, then, the Euler--Arnold equation of the metric $G^A$ is given by:
\begin{equation}\label{eq:Euler--Arnold}
  m_t + um_x + 2m u_x = 0, \qquad m = Au, \qquad u(0) = u_0\in C^{\infty}(S^1).
\end{equation}
In this equation, $m$ is the so-called momentum associated to the velocity $u$. If the inertia operator $A$ is a ``\emph{nice operator}'' of order $s>\frac32$, the global existence of smooth solutions to this equation has been shown by Escher and Kolev in~\cite{EK2014}.

\section{The right-invariant $H^{3/2}$-metric on $\Diff(S^1)$}
\label{sec:Fourier-multipliers}

We will now formally introduce the class of \emph{nice operators} we are interested in. We will assume that $A$ is a continuous linear operator on $C^{\infty}(S^1)$ that commutes with differentiation. In that case, $A$ is a \emph{Fourier-multiplier}, \textit{i.e.}
\begin{align}
  (Au)(x) = \sum_{k\in\mathbb Z} a(k) u_{k} \operatorname{exp}(ikx),
\end{align}
where $u_{k}$ is the $k$-th Fourier coefficients of the vector field $u$, see~\cite[App. A]{EK2014}. The sequence $a:\mathbb Z \rightarrow \mathbb C$ is called the \emph{symbol} of $A$ and we will use the notation $A=\operatorname{op}(a(k))$ or equivalently $A=a(D)$. For a more detailed introduction to the theory of Fourier multipliers in the context of $\Diff(S^1)$ we refer to the article \cite{EK2014}. The most important example in our context is the inertia operator for the fractional order Sobolev metric $H^{s}$, which reads:
\begin{equation}\label{eq:Sobolev-operator}
  A = \operatorname{op}\left((1+k^{2})^{s/2}\right).
\end{equation}
In this article we will be interested in the specific case $s=3$.

\begin{definition}
  Given $r \in \RR$, a Fourier multiplier $a(D)$ is of class $\mathcal{S}^{r}$ iff $a$ extends to a smooth function $\RR^{d} \to \CC^{d}$ and satisfies moreover the following condition:
  \begin{equation*}
    a^{(l)}(\xi) = O(\abs{\xi}^{r-l}), \qquad \forall l \in \mathbb N.
  \end{equation*}
\end{definition}

\begin{remark}
  Note that a Fourier multiplier $a(D)$ of class $\mathcal{S}^{r}$ extends to a \emph{bounded linear operator}
  \begin{equation*}
    \HH{q} \to \HH{q-r}
  \end{equation*}
  for any $q \ge r$.
\end{remark}

\begin{definition}
  If a Fourier multiplier $a(D)$ extends to a bounded linear operator
  \begin{equation*}
    \HH{q} \to \HH{q-r}
  \end{equation*}
  for $q$ big enough and $r \in \RR$, we will say that $a(D)$ is \emph{of order less than or equal to $r$}.
\end{definition}

We will impose a slightly more restrictive condition on the symbol class, by requiring that $a$ has a series representation of the form
\begin{align}\label{eq:symbol-expansion}
  a(\xi) = \sum_{k=0}^{\infty} a_{3-k} \abs{\xi}^{3-k}, \qquad a_3 \ne 0.
\end{align}
In terms of the operator $A=a(D)$ this translates to
\begin{equation}
  A = a_{3} (HD)^3+R_{2}= a_{3} (HD)^3+a_{2}(HD)^{2} +R_1,
\end{equation}
where
\begin{equation}
  R_{k}:= \sum_{j=-\infty}^k a_j (HD)^j
\end{equation}
and $H := \operatorname{op}(-i\sgn(k))$ denotes the Hilbert transform. Using the property $H^{2}=-\operatorname{Id}$, we can then rewrite $A$ to obtain
\begin{equation}\label{eq:A-expansion}
  A = -a_{3}HD^3 + R_{2} = - a_{3}HD^3 - a_{2}D^{2} + R_1.
\end{equation}
Note, that the remainder terms $R_{k}$ are Fourier multipliers of order at most $k$.

Finally, in order to ensure local well-posedness of the Euler--Arnold equation~\eqref{eq:Euler--Arnold} (see~\cite{EK2014}), we will require, furthermore, an ellipticity condition on $A$.

\begin{definition}
  A Fourier multiplier $A=a(D)$ in the class $\mathcal{S}^{r}$ is called \emph{elliptic} if
  \begin{equation*}
    \left( 1 + \abs{\xi}^{2}\right)^{r/2} \lesssim \abs{a(\xi)}, \qquad \forall \xi \in \RR^{d}.
  \end{equation*}
\end{definition}

In the following definition, we summarize the assumptions on the class of operators we will consider.

\begin{definition}\label{operatordef}
  An operator $A \in L(C^{\infty}(S^1))$ is in the class $\mathcal E^3_{\operatorname{cl}}$ iff the following conditions are satisfied:
  \begin{enumerate}
    \item $A$ is a Fourier multiplier of class $\mathcal S^3$;
    \item $A$ is elliptic;
    \item $a(\xi)$ is real for all $\xi \in \mathbb R$;
    \item $a(\xi)$ is positive for all $\xi \in \mathbb R$;
    \item $a(\xi)$ has a series expansion of the form \eqref{eq:symbol-expansion}.
  \end{enumerate}
\end{definition}

\begin{remark}
  Note that assumption (1) guarantees that the operator is of order three and commutes with differentiation; The ellipticity condition (2) is required to show local well-posedness of the geodesic equation. Condition (3) guarantees that $A$ is $L^{2}$-self-adjoint and (4) that it is a positive definite operator. Assumption (5) is a technical condition, that is essential for our long-time existence proof.
\end{remark}

\begin{example}\label{ex:H32-metric}
  A trivial example for an operator within the class $\mathcal S^3_{\operatorname{cl}}$ is the operator
  $A=1-H\partial_{\theta}^3$. Another example consist of the Sobolev metric of fractional order $\frac32$ as defined
  in \eqref{eq:Sobolev-operator}. The ellipticity of this inertia operator has been shown in \cite{EK2014}. To see that this operator satisfies assumption (5), one only needs to make a series expansion of the Fourier multiplier $(1+k^{2})^{\frac32}$.
\end{example}

\begin{example}\label{ex:WPmetric}
  The Weil-Petersson metric comes from $A = (HD)^3 - (HD) = -H(D^3+D)$. This satisfies all conditions of Definition \ref{operatordef} except ellipticity, since it is degenerate on the three-dimensional subspace spanned by $\{1, \cos{x}, \sin{x}\}$. The metric generated by $A$ is not Riemannian on $\mathcal{D}(S^1)$ because of this degeneracy, but it descends to a Riemannian metric on the universal Teichm\"uller space, which is the closure of    $\mathcal{D}(S^1)/PSL_2(\mathbb{R})$. Geodesics in the quotient which are initially $C^{\infty}$ will remain so for all time by the result of \cite{PW2016}.
\end{example}

\section{Global Well-Posedness of the EPDiff equation.}
\label{sec:proof}

In this section, we will prove the global existence of solutions to the EDDiff equation of metrics of order $3/2$:

\begin{theorem}\label{thm:maintheorem2}
  Let $G$ be the right invariant metric on $\Diff(S^1)$ with inertia operator $A$ in the class
  $\mathcal E^3_{\operatorname{cl}}$ and let $u_0$ be an $H^s$ velocity field on $S^1$, for some $s>\frac32$. Then the solution $u(t)$ of the Euler--Arnold-equation~\eqref{eq:Euler--Arnold} of the metric $G$ with $u(0)=u_0$ remains in $H^s$ for all time. In particular if $u_0$ is $C^{\infty}$ then so is $u(t)$ for all $t>0$. Thus the space
  $\left(\Diff(S^1),G\right)$ is geodesically complete.
\end{theorem}

Note, that this result implies in particular the completeness of the $H^{3/2}$-metric, as the inertia operator  $A = \operatorname{op}\left((1+k^{2})^s\right)$ is of class $\mathcal E^3_{\operatorname{cl}}$, c.f. Example~\ref{ex:H32-metric}.

\begin{remark}
  The condition $u_0\in H^s$ for some $s>\frac32$ is necessary as our proof relies on the velocity field being of class $C^1$. The later fact is based on a full investigation of local well-posedness of geodesic flows for $H^s$ metrics (for $s \ge 1/2$) on Banach manifolds $\D{q}$ (when $q > \frac32$) which was achieved in~\cite{EK2014}. When $q \le \frac32$, $\D{q}$ is not even a group since composition is not stable, and then, the methods used in~\cite{EK2014} do not work anymore. For the metric corresponding to the Euler-Weil-Petersson equation, which corresponds to a special case of the critical index $s = \frac32$, Gay-Balmaz and Ratiu~\cite{GAYBALMAZ2015717} have developed a different framework, emphasising weak solutions rather than smooth solutions. They are thus able to consider the case where $u\in H^{3/2}$, but this approach does not seem to give information about whether initial conditions of higher regularity give solutions of higher regularity in the long term.
\end{remark}

Before we are able to give the proof of our main result we will need to collect several technical estimates that will be necessary to achieve our result.

\begin{lemma}\label{lem:mainlemma}
  Let $u,v\colon S^1\to\RR$ be smooth functions on $S^1$. We have the following estimates:
  \begin{enumerate}
    \item[(a)]  Let $F(u) := uu_{xx} + H(uHu_{xx})$. Then
          \begin{equation*}
            \norm{F(u)}_{L^{\infty}} \le 4 \norm{u}_{H^{3/2}}^{2}.
          \end{equation*}
    \item[(b)] Let $G(u) := B(u_{x}Hu_{x})$, where
          $B$ is a Fourier multiplier of order less than or equal to $-2$. Then
          \begin{equation*}
            \norm{G(u)}_{L^{\infty}} \le C \norm{u}^{2}_{{H}^{1/2}},
          \end{equation*}
          for some constant $C>0$.
    \item [(c)]  Let $B$ be a Fourier multiplier of order $s < -1/2$. Then
          \begin{align*}
            \norm{Bu}_{L^{\infty}} \leq C \norm{u}_{L^{2}},
          \end{align*}
          for some constant $C>0$.
    \item [(d)] Let $B_1$ be Fourier multiplier operator of order $k_{1} \le 3/2$ and $B_{2}$ be Fourier multiplier operator of order $k_{2} < -1/2$. Then, we have
          \begin{equation*}
            \norm{B_{2}(uB_1v)}_{L^{\infty}} \le C \norm{u}_{H^{3/2}}\norm{v}_{H^{3/2}},
          \end{equation*}
          for some constant $C>0$.
  \end{enumerate}
\end{lemma}

\begin{proof}
  The proof of item (a) can can be found in~\cite{BKP2016}. The main difficulty to show this statement is to prove that
  \begin{equation*}
    F(u)(x) = 2 \sum_{n=1}^{\infty}(2n-1)\bigg\lvert \sum_{k=n}^{\infty}u_{k}e^{ikx}\bigg\rvert^{2}.
  \end{equation*}
  where $u_{k}$ are the Fourier coefficients of $u$, i.e., $u(x) = \sum_{k \in \ZZ} u_{k}e^{ikx}$,
  see~\cite[Theorem 17]{BKP2016}. From this expression for $F$ the estimate follows by direct calculation.

  To show statement (b) we will slightly adapt the proof of the related result in the special case $B = -(\partial_x^2)^{-1}$~\cite[Theorem 8]{PW2016}. In fact the first part of their proof, can be used word by word.
  Therefore we express $u$ in a Fourier basis $u(x) = \sum_{n \in \ZZ} u_{n}e^{inx}$, and let $h = u_{x}Hu_{x}$. Then we have
  \begin{align*}
    (u_{x}Hu_{x})(x) & = i \sum_{m,n\in\ZZ} mn \, u_{m} u_{n} (\sgn{n}) e^{i(m+n)x}                          \\
                     & = i \sum_{k\in \ZZ} \left(\sum_{n\in \ZZ} \abs{n} (k-n)u_{k-n} u_{n}  \right) e^{ikx} \\
                     & = i\sum_{k\in \ZZ} h_{k} e^{ikx},
  \end{align*}
  where $h_{k} = \sum_{n\in\ZZ} \abs{n} (k-n) u_{k-n}u_{n}$. Now let us simplify $h_{k}$. For $k>0$, we have
  \begin{align*}
    h_{k} & = \sum_{n=1}^{\infty} n(k-n) u_{n} u_{k-n} + \sum_{n=1}^{\infty} n(k+n) \overline{u_{n}} u_{k+n}                                                        \\
          & = \sum_{n=1}^{k-1} n(k-n) u_{n} u_{k-n} + \sum_{m=1}^{\infty} (k+m)(-m) u_{k+m} \overline{u_{m}} + \sum_{n=1}^{\infty} n(k+n) \overline{u_{n}} u_{k+n},
  \end{align*}
  where we used the substitution $m=n-k$. Clearly the middle term cancels the last term, so
  \begin{equation*}
    h_{k} = \sum_{n=1}^{k-1} n(k-n) u_{n} u_{k-n}.
  \end{equation*}
  It is easy to see that $h_0=0$ due to cancellations, while if $k<0$, we get
  \begin{equation*}
    h_{k} = -\sum_{n=1}^{\abs{k}-1} n(\abs{k}-n) \overline{u_{n}} \overline{u_{\abs{k}-n}} = - \overline{h_{\abs{k}}}.
  \end{equation*}
  Note in particular that $h_1=h_{-1}=0$. We thus obtain
  \begin{equation*}
    (u_{x}Hu_{x})(x) = \sum_{k=2}^{\infty} \left( ih_{k} e^{ikx} - i\overline{h_{k}} e^{-ikx}\right).
  \end{equation*}
  From here we slightly differ from the proof of \cite{PW2016}, although the idea remains the same. Applying $B$ to the function $h$ yields
  \begin{equation*}
    G(u)(x) = B(h)(x) = \sum_{k=2}^{\infty} \left( ib(k)h_{k} e^{ikx} - ib(-k)\overline{h_{k}} e^{-ikx}\right).
  \end{equation*}
  We estimate only the first part of the sum, the second is similar:
  \begin{align*}
    \norm{G(u)}_{L^{\infty}} & \le \sum_{k=2}^{\infty} \sum_{n=1}^{k-1} b(k) n(k-n) \abs{u_{n}} \abs{u_{k-n}}  \\
                             & = \sum_{n=1}^{\infty} \sum_{k=n+1}^{\infty}b(k)n(k-n) \abs{u_{n}} \abs{u_{k-n}} \\
                             & = \sum_{n=1}^{\infty} \sum_{m=1}^{\infty} b(n+m)nm \abs{u_{n}} \abs{u_{m}}.
  \end{align*}
  Using the assumption on the symbol of $B$, we then have
  \begin{align*}
    \norm{G(u)}_{L^{\infty}} & \le C\sum_{n=1}^{\infty} \sum_{m=1}^{\infty} \frac{nm}{1+(m+n)^{2}} \abs{u_{n}} \abs{u_{m}}     \\
                             & \le 2C \sum_{n=1}^{\infty} \sum_{m=1}^{\infty} \frac{\sqrt{nm} \abs{u_{n}} \abs{u_{m}}}{n+m}    \\
                             & \le 2C\pi \left( \sum_{n=1}^{\infty} n\abs{u_{n}}^{2}\right) \le C\pi \norm{u}^{2}_{{H}^{1/2}},
  \end{align*}
  where the inequality in the last line is precisely the well-known Hilbert double series theorem.

  For statement (c) we use the fact that since $B$ is of order $s < -1/2$, there exists a constant $\tilde{C}$ such that
  \begin{equation*}
    b(m) \le \tilde{C} (1+m^{2})^{s/2}
  \end{equation*}
  Now we express $u$ again in a Fourier basis $u(x) = \sum_{n \in \ZZ} u_{n}e^{inx}$. Then, we have
  \begin{equation*}
    (Bu)(x) = \sum_{m\in \ZZ} b(m) u_{m} e^{imx},
  \end{equation*}
  and thus
  \begin{align*}
    \norm{Bu}_{L^{\infty}} & \le  \sum_{m\in \ZZ} \abs{b(m)u_{m}}                                            \\
                           & \le \tilde{C} \sum_{m\in \ZZ}(1+m^{2})^{s/2}\abs{u_{m}}                         \\
                           & \le \tilde{C} \bigg(\sum_{m\in \ZZ} (1+m^{2})^{s} \bigg)^{1/2} \norm{u}_{L^{2}} \\
                           & \le C \norm{u}_{L^{2}}.
  \end{align*}

  Finally for statement (d), using that $B_{2}$ is of order $k_{2} < -1/2$ we get  by virtue of statement (c),
  \begin{equation*}
    \norm{B_{2}(uB_1v)}_{L^{\infty}} \le C_1 \norm{uB_1v}_{L^{2}}
  \end{equation*}
  Now, we will recall the following inequality (see~\cite[Lemma 2.3]{IKT2013}) on pointwise multiplication in Sobolev spaces, valid for $q >1/2$ and $0 \le \rho \le q$:
  \begin{equation*}
    \norm{uw}_{H^{\rho}} \lesssim \norm{u}_{H^{q}} \norm{w}_{H^{\rho}}.
  \end{equation*}
  We deduce from it, using $q = 3/2$ and $\rho = 0$, that
  \begin{equation*}
    \norm{uB_1v}_{L^{2}} \lesssim \norm{u}_{H^{3/2}} \norm{B_1v}_{L^{2}} \lesssim \norm{u}_{H^{3/2}} \norm{v}_{H^{k_{1}}} \lesssim \norm{u}_{H^{3/2}} \norm{v}_{H^{3/2}},
  \end{equation*}
  which achieves the proof.
\end{proof}

\begin{proof}[Proof of Theorem~\ref{thm:maintheorem2}]
  According to~\cite[Theorem 5.6]{EK2014a}, we only need to show that for any solution $u(t)$ of the Euler--Arnold equation~\eqref{eq:Euler--Arnold}, the norm $\norm{u_{x}(t)}_{L^{\infty}}$ is \emph{bounded on every bounded time interval}.

  Let $\varphi(t)$ be the flow of the time dependent vector field $u(t)$, we have
  \begin{equation*}
    \partial_{t} (u_{x} \circ \varphi) = (u_{tx} + uu_{xx}) \circ \varphi.
  \end{equation*}
  Now, from~\eqref{eq:Euler--Arnold}, we get
  \begin{align*}
    u_{tx} & = -A^{-1}D(uAu_{x} + 2u_{x}Au)                 \\
           & = -A^{-1}\left(D^{2}(uAu) + D(u_{x}Au)\right).
  \end{align*}
  Thus, observing that $\norm{w}_{L^{\infty}} = \norm{w \circ \eta}_{L^{\infty}}$, for every $w \in \CS$ and $\eta \in \Diff(S^{1})$, we have
  \begin{equation*}
    \norm{u_{x}(t)}_{L^{\infty}} \le \norm{u_{x}(0)}_{L^{\infty}} + \int_{0}^{t} \norm{Q(u(s))}_{L^{\infty}}\, ds
  \end{equation*}
  where
  \begin{equation*}
    Q(u) := uu_{xx} - A^{-1}\left(D^{2}(uAu) + D(u_{x}Au)\right).
  \end{equation*}
  Therefore, thanks to the Grönwall inequality, it is sufficient to show that
  \begin{equation*}
    \norm{Q(u)}_{L^{\infty}} \le \alpha \norm{u}^{2}_{H^{3/2}} + \beta \norm{u}_{H^{3/2}}\norm{u_{x}}_{L^{\infty}},
  \end{equation*}
  for some positive constants $\alpha,\beta$, because the norm $\norm{u}_{H^{3/2}}$ is equivalent to the norm $\norm{u}_{A}$, given by
  \begin{equation*}
    \norm{u}_{A}^{2} := \int_{S^{1}} uAu \, dx,
  \end{equation*}
  which is an integral constant. In the following we will achieve the remaining estimate for $Q(u)$.

  To prove the bound for $Q(u)$ we will use the decomposition~\eqref{eq:A-expansion} of $A$ to further expand $Q(u)$ in a sum of terms that can be bounded separately. We have:
  \begin{align*}
    D^{2}(uAu) + D(u_{x}Au) & = - a_{3} D^{2}(uHu_{xxx}) + D^{2}(uR_{2}u) - a_{3}  D(u_xHu_{xxx}) + D(u_xR_{2}u)             \\
                            & = - a_{3} D^3(uHu_{xx}) + a_{3} D^{2}(u_xHu_{xx}) + D^{2}(uR_{2}u)                             \\
                            & \quad - a_{3} D^{2}(u_xHu_{xx}) + a_{3} D(u_{xx}Hu_{xx}) + D(u_xR_{2}u)                        \\
                            & = - a_{3} D^3(uHu_{xx}) + D^{2}(uR_{2}u) + a_{3} D(u_{xx}Hu_{xx}) + D(u_xR_{2}u)               \\
                            & = a_{3} HD^3H(uHu_{xx}) + D^{2}(uR_{2}u) + a_{3} D(u_{xx}Hu_{xx}) + D(u_xR_{2}u)               \\
                            & = - AH(uHu_{xx}) + a_{3}  D(u_{xx}Hu_{xx}) + R_{2}H(uHu_{xx}) + D^{2}(uR_{2}u) + D(u_xR_{2}u).
  \end{align*}
  Further expanding with $R_{2} = -a_{2} D^{2} + R_1$, we get:
  \begin{align*}
    R_{2}H(uHu_{xx}) & = - a_{2} D^{2}H(uHu_{xx}) + R_1H(uHu_{xx})                      \\
                     & = - a_{2} D^{2}H(uHu_{xx}) + R_1HD(uHu_{x}) - R_1H(u_{x}Hu_{x}), \\
    D^{2}(uR_{2}u)   & = - a_{2} D^{2}(uu_{xx}) + D^{2}(uR_1u),                         \\
    D(u_xR_{2}u)     & = - a_{2} D(u_xu_{xx}) + D(u_xR_1u)                              \\
                     & = - \frac{a_{2}}{2} D^{2}(u_x^{2}) + D(u_xR_1u).
  \end{align*}
  Summing up and rearranging all the terms, we get finally:
  \begin{equation*}
    Q(u) = \sum_{i=1}^8 Q_{i}(u),
  \end{equation*}
  where
  \begin{align*}
    Q_{1}(u) & = H(uHu_{xx}) + u u_{xx},                        & Q_{2}(u) & = - a_{3}  A^{-1}D(u_{xx}Hu_{xx}),     \\
    Q_{3}(u) & = a_{2}A^{-1}D^{2}\big(H(uHu_{xx})+uu_{xx}\big), & Q_{4}(u) & = \frac{a_{2}}{2}A^{-1}D^{2}(u_x^{2}), \\
    Q_{5}(u) & = -A^{-1}D^{2}(uR_1u),                           & Q_{6}(u) & = -A^{-1}R_1HD(uHu_{x}),               \\
    Q_{7}(u) & = A^{-1}R_1H(u_xHu_{x}),                         & Q_{8}(u) & = -A^{-1}D(u_xR_1u).
  \end{align*}
  To achieve the proof of Theorem~\ref{thm:maintheorem2}, it only remains to show that all the quadratic terms $Q_{i}(u)$ are bounded either by $\norm{u}_{H^{3/2}}^{2}$ or by $\norm{u}_{H^{3/2}}\norm{u_{x}}_{L^{\infty}}$.

  In particular we will show that for $i = 1, \dotsc , 8$ there exists $\kappa_{i} >0$ such that:
  \begin{equation*}
    \norm{Q_{i}(u)}_{L^{\infty}} \le \kappa_{i} \norm{u}_{H^{3/2}}^{2}, \qquad \text{for} \quad i=1,2,3,5,6,7,
  \end{equation*}
  and
  \begin{equation*}
    \norm{Q_{i}(u)}_{L^{\infty}} \le \kappa_{i} \norm{u_{x}}_{L^{\infty}}\norm{u}_{H^{3/2}}, \qquad \text{for} \quad i=4,8.
  \end{equation*}

  Using statement (a) of Lemma~\ref{lem:mainlemma}, we can bound the supremum norm of $Q_{1}$ via
  \begin{equation*}
    \norm{Q_{1}(u)}_{L^{\infty}} \le \kappa_{1} \norm{u}_{H^{3/2}}^{2}.
  \end{equation*}
  Using statement (b) of Lemma~\ref{lem:mainlemma}, with $B=A^{-1}D$, we obtain:
  \begin{equation*}
    \norm{Q_{2}(u)}_{L^{\infty}} \leq \kappa_{2} \norm{u}^{2}_{{H}^{3/2}} \qquad \text{and} \qquad \norm{Q_{7}(u)}_{L^{\infty}} \lesssim  \norm{u}^{2}_{{H}^{1/2}} \leq \kappa_{7} \norm{u}^{2}_{{H}^{3/2}}.
  \end{equation*}
  Since $A^{-1}D^{2}$ is of order $\le -1$, we can use statement (c) of Lemma~\ref{lem:mainlemma} to bound $Q_{3}$ via:
  \begin{equation*}
    \norm{Q_{3}(u)}_{L^{\infty}} \lesssim \norm{\big(H(uHu_{xx})+uu_{xx}\big)}_{L^{2}} \le \norm{\big(H(uHu_{xx})+uu_{xx}\big)}_{L^{\infty}} \le \kappa_{3} \norm{u}^{2}_{{H}^{3/2}},
  \end{equation*}
  the last inequality following again from statement (a) of Lemma~\ref{lem:mainlemma}. For $Q_{5}$, we apply statement (d) of Lemma~\ref{lem:mainlemma} with $B_1=R_1$ of order $\le 1$ and $B_{2}=A^{-1}D^{2}$ of order $\le -1$, which yields $\norm{Q_{5}(u)}_{L^{\infty}}\leq \kappa_{3} \norm{u}^{2}_{{H}^{3/2}}$. Finally, using again statement (c) of Lemma~\ref{lem:mainlemma}, we get
  \begin{equation*}
    \norm{Q_{4}(u)}_{L^{\infty}} \lesssim \norm{u_{x}^{2}}_{L^{2}} \le \norm{u_{x}}_{L^{\infty}}\norm{u_{x}}_{L^{2}} \lesssim \norm{u_{x}}_{L^{\infty}}\norm{u}_{H^{3/2}},
  \end{equation*}
  then
  \begin{equation*}
    \norm{Q_{6}(u)}_{L^{\infty}} \lesssim \norm{uHu_{x}}_{L^{2}}\lesssim \norm{u}_{L^{\infty}}\norm{Hu_{x}}_{L^{2}}\lesssim \norm{u}_{H^{3/2}}^{2},
  \end{equation*}
  and
  \begin{equation*}
    \norm{Q_{8}(u)}_{L^{\infty}} \lesssim \norm{u_xR_1u}_{L^{2}}\lesssim \norm{u_{x}}_{L^{\infty}}\norm{R_1u}_{L^{2}}\lesssim \norm{u_{x}}_{L^{\infty}}\norm{u}_{H^{3/2}},
  \end{equation*}
  which concludes the proof.
\end{proof}


\end{document}